\documentclass{amsart}
\usepackage{setspace}
\usepackage{a4}
\usepackage{amssymb,amsmath,amsthm,latexsym}
\usepackage{amsfonts}
\usepackage{amsfonts}
\usepackage{graphicx}
\usepackage{textcomp}
\usepackage{cite}
\usepackage{enumerate}
\usepackage[mathscr]{euscript}
\usepackage{mathtools}
\newtheorem{theorem}{Theorem}[section]

\newtheorem{definition}[theorem]{Definition}

\newtheorem{lemma} [theorem]{Lemma}

\newtheorem{proposition}[theorem]{Proposition}
\newtheorem{remark}[theorem]{Remark}
\title{This is the title}
\usepackage{amssymb}
\usepackage{amssymb}
\usepackage{amssymb}
\usepackage{amssymb}
\usepackage{amsmath}
\usepackage{tikz}
\usepackage{hyperref}
\usepackage{enumerate}
\usepackage{mathtools}
\usepackage{amsmath}
\usepackage{tikz}
\usepackage{amssymb}
\usepackage{amsmath}
\usepackage{tikz}
\usepackage{hyperref}

\begin{document}
\begin{center}
{\bf{TOWARDS CHARACTERIZATIONS OF  APPROXIMATE SCHAUDER FRAME AND ITS DUALS FOR BANACH SPACES}}\\
K. MAHESH KRISHNA AND P. SAM JOHNSON  \\
Department of Mathematical and Computational Sciences\\ 
National Institute of Technology Karnataka (NITK), Surathkal\\
Mangaluru 575 025, India  \\
Emails: kmaheshak@gmail.com,  sam@nitk.edu.in\\

Date: \today
\end{center}

\hrule
\vspace{0.5cm}
%--------------------------------------
\textbf{Abstract}: We begin the study of characterizations of recently defined  approximate Schauder frame (ASF) and its duals for separable Banach spaces. We show that, under some conditions,  both ASF and its dual frames can be characterized for Banach spaces. We also give an operator-theoretic characterization for similarity of ASFs. Our results encode the results of Holub, Li,  Balan, Han, and Larson. We also address orthogonality of ASFs.



%\tableofcontents

\section{Introduction}
Let $\mathcal{H}$ be a separable Hilbert space. Classical definition of a frame for $\mathcal{H}$ as set out by Duffin and Schaeffer, in 1952 \cite{DUFFIN} reads as follows.
\begin{definition}\cite{DUFFIN}
A sequence $\{\tau_n\}_n$ in  $\mathcal{H}$ is said to be a
frame for $\mathcal{H}$ if there exist $a,b>0$ such that 
\begin{align}\label{FRAMEINEQUALITY}
a\|h\|^2 \leq \sum_{n=1}^\infty |\langle h, \tau_n\rangle|^2\leq b\|h\|^2, \quad \forall h \in \mathcal{H}.
\end{align}
Constants $a$ and $b$ are called as lower and upper frame bounds, respectively. If $a$ can take the value 0, then we say that $\{\tau_n\}_n$ is a Bessel sequence 
 for $\mathcal{H}$. 
\end{definition}
Given a collection $\{\tau_n\}_n$, in general, it is difficult to find $a$ and $b$ such that the two inequalities in (\ref{FRAMEINEQUALITY}) hold. Therefore, it is natural to ask whether there is a characterization for frame without using frame bounds. Orthonormal bases are the simplest and easiest sequences we can handle in a Hilbert space, so one attempts to obtain characterization using orthonormal bases. Since every separable Hilbert space is isometrically isomorphic to the standard Hilbert space $\ell^2(\mathbb{N})$ and the standard unit vectors $\{e_n\}_n$ form an orthonormal basis for $\ell^2(\mathbb{N})$, one can further ask whether frames can be characterized using $\{e_n\}_n$. This question was answered affirmatively by Holub in 1994 \cite{HOLUB} as follows.
\begin{theorem}\cite{HOLUB}\label{HOLUBTHEOREM}
A sequence $\{\tau_n\}_n$ in  $\mathcal{H}$ is a
frame for $\mathcal{H}$	if and only if there exists a surjective bounded linear operator $T:\ell^2(\mathbb{N}) \to \mathcal{H}$ such that $Te_n=\tau_n$, for all $n \in \mathbb{N}$.
\end{theorem}
There is a slight variation of Theorem \ref{HOLUBTHEOREM} given by Christensen. 
\begin{theorem}\cite{OLEBOOK}\label{OLECHA}
Let $\{\omega_n\}_n$ be an orthonormal basis for   $\mathcal{H}$. Then  a sequence $\{\tau_n\}_n$ in  $\mathcal{H}$ is a
frame for $\mathcal{H}$	if and only if there exists a surjective bounded linear operator $T:\mathcal{H} \to \mathcal{H}$ such that $T\omega_n=\tau_n$, for all $n \in \mathbb{N}$.	
\end{theorem}
Now we look into some fundamental properties of frames which are exhibited in the following theorem. 
\begin{theorem}\cite{DUFFIN, HANMEMOIRS}\label{FFF}
	Let $\{\tau_n\}_n$ be a frame for $\mathcal{H}$. Then
	\begin{enumerate}[\upshape(i)]
	\item The frame operator 
		$
		S_\tau :\mathcal{H} \ni h \mapsto \sum_{n=1}^\infty \langle h, \tau_n\rangle\tau_n\in
		\mathcal{H}
	$
	is  a well-defined  bounded linear, positive and invertible operator. Further, 
	\begin{align}\label{REP}
h=\sum_{n=1}^\infty \langle h, S_\tau^{-1}\tau_n\rangle \tau_n=\sum_{n=1}^\infty
		\langle h, \tau_n\rangle S_\tau^{-1}\tau_n, \quad \forall h \in
		\mathcal{H}.
		\end{align}
	\item $\{S_\tau^{-1}\tau_n\}_n$ is a frame for $\mathcal{H}$.
		\item The analysis operator 
		$
		\theta_\tau: \mathcal{H} \ni h \mapsto \{\langle h, \tau_n\rangle \}_n \in \ell^2(\mathbb{N})
		$
		is  a well-defined  bounded linear, injective operator.
		\item   Adjoint of $\theta_\tau$ 
		(synthesis operator) is given by 
	$
		\theta_\tau^*: \ell^2(\mathbb{N}) \ni \{a_n \}_n \mapsto \sum_{n=1}^\infty a_n\tau_n \in \mathcal{H}
		$
		which is surjective.
		\item Frame operator 
		splits as $S_\tau=\theta_\tau^*\theta_\tau.$
	\item $ P_\tau \coloneqq \theta_\tau S_\tau^{-1} \theta_\tau^*:\ell^2(\mathbb{N}) \to \ell^2(\mathbb{N})$ is an orthogonal  projection onto $ \theta_\tau(\mathcal{H})$.
\end{enumerate}

\end{theorem}
Previous theorem says that, given a vector $ h \in \mathcal{H}$, we need to consider sequences $\{\tau_n\}_n$ and $\{S_\tau^{-1}\tau_n\}_n$ to construct $h$ using Equation (\ref{REP}). It is well-known in the theory of frames that, in general, there are frames  $\{\omega_n\}_n$ for $  \mathcal{H}$ which will also give representation of element as given by $\{S_\tau^{-1}\tau_n\}_n$, namely, 
\begin{align}\label{DUAL}
h=\sum_{n=1}^\infty \langle h, \omega_n\rangle \tau_n=\sum_{n=1}^\infty
\langle h, \tau_n\rangle \omega_n, \quad \forall h \in
\mathcal{H}.
\end{align}
Equation (\ref{DUAL}) leads to the notion of dual frames as stated below. 
\begin{definition}\cite{LI}
Let $\{\tau_n\}_n$ be a frame for $\mathcal{H}$. A  frame  $\{\omega_n\}_n$ for $  \mathcal{H}$ is said to be a dual frame for $\{\tau_n\}_n$ if Equation (\ref{DUAL}) holds.
\end{definition}
Just like characterization of frames, given a frame,  we  seek a description of each of its dual frame. This problem was solved by Li in 1995 \cite{LI} in the following two lemmas and a theorem. 
\begin{lemma}\cite{LI}\label{LILEMMA1}
Let $\{\tau_n\}_n$ be a frame for $\mathcal{H}$ and $\{e_n\}_n$ be the standard orthonormal basis for 	$\ell^2(\mathbb{N})$. Then a frame $\{\omega_n\}_n$ is a dual frame for $\{\tau_n\}_n$ if and only if 
\begin{align*}
\omega_n=U e_n, \quad \forall n \in \mathbb{N},
\end{align*}
where $U:\ell^2(\mathbb{N})\to \mathcal{H}$ is a bounded left-inverse of $\theta_\tau$.
\end{lemma}
\begin{lemma}\cite{LI}\label{LILEMMA2}
Let $\{\tau_n\}_n$ be a frame for $\mathcal{H}$. Then $L:\ell^2(\mathbb{N})\to \mathcal{H}$ 	is a bounded left-inverse of $\theta_\tau$ if and only if 
\begin{align*}
L=S_\tau^{-1}\theta_\tau^*+V(I_{\ell^2(\mathbb{N})}-\theta_\tau S_\tau^{-1} \theta_\tau^*),
\end{align*}
where $V:\ell^2(\mathbb{N})\to \mathcal{H}$ is a bounded operator.
\end{lemma}
\begin{theorem}\cite{LI}\label{LITHM}
Let $\{\tau_n\}_n$ be a frame for $\mathcal{H}$. Then a frame $\{\omega_n\}_n$ is a dual frame for $\{\tau_n\}_n$ if and only if 
\begin{align*}
\omega_n=S_\tau^{-1} \tau_n+\rho_n-\sum_{k=1}^{\infty}\langle S_\tau^{-1} \tau_n, \tau_k\rangle \rho_k, \quad \forall  n \in \mathbb{N}, 
\end{align*}
where 	$\{\rho_n\}_n$ is a Bessel sequence for $\mathcal{H}$.
\end{theorem}
We  can  see proofs of Lemmas \ref{LILEMMA1}, \ref{LILEMMA2} and Theorem \ref{LITHM} in the excellent book \cite{OLEBOOK}. We again consider the frame $\{S_\tau^{-1}\tau_n\}_n$. Note that this frame is obtained by the action of an invertible operator $S_\tau^{-1}$ to the original frame $\{\tau_n\}_n$. This leads to the question: what are all the frames which are obtained by operating an invertible operator to the  given frame? This naturally brings us to the following definition. 
\begin{definition}\cite{BALAN}\label{SIMILARDEFHILBERT}
Two frames $\{\tau_n\}_n$ and $\{\omega_n\}_n$ for $  \mathcal{H}$ are said to be similar or equivalent if there exist bounded invertible operator $T:\mathcal{H} \to \mathcal{H}$ such that
\begin{align}\label{SIMILARITYEQUATION}
\omega_n=T \tau_n, \quad\forall n \in \mathbb{N}.
\end{align} 
\end{definition}
Given frames $\{\tau_n\}_n$ and $\{\omega_n\}_n$, it is rather difficult to check whether they are similar because one has to get an invertible operator and verify Equation (\ref{SIMILARITYEQUATION}) for every natural number. Thus it is better if there is a characterization which does not  involve natural numbers and involves only operators. Further, it is natural to ask whether there is a formula for the operator $T$ which gives similarity. This was done by Balan in 1999 and independently by Han, and Larson in 2000 which states as follows. 
\begin{theorem}\cite{BALAN, HANMEMOIRS}\label{BALANCHARSIM}
For two frames $\{\tau_n\}_n$ and $\{\omega_n\}_n$ for $  \mathcal{H}$, the following are equivalent.	
	\begin{enumerate}[\upshape (i)]
		\item $\{\tau_n\}_n$ and $\{\omega_n\}_n$ are similar, i.e., there exists bounded invertible operator $T:\mathcal{H} \to \mathcal{H}$ such that $\omega_n=T \tau_n$, $\forall n \in \mathbb{N}$.
		\item $\theta_\omega=\theta_\tau T$,  for some bounded invertible operator $T:\mathcal{H} \to \mathcal{H}$.
		\item $ P_\omega=P_\tau$.
	\end{enumerate}
If one of the above conditions is satisfied, then the invertible operator in (i) and (ii) is unique and is given by $T=S_\tau^{-1}\theta_\tau^*\theta_\omega$.
\end{theorem}
In this paper, we try to obtain characterizations for a class of ASFs (see Section \ref{SECTIONTWO} for definitions), thus initiating the study of characterizations for arbitrary ASFs. Hilbert space frame theory is more fertile due to the fact that we can continuously switch between $\mathcal{H}$ and the standard separable Hilbert space $\ell^2(\mathbb{N})$ (see Theorem \ref{FFF}). ASFs do not have this property. For this reason, we first define p-ASF in Definition \ref{PASFDEF} which is inspired from Theorem \ref{FFF}. As an immediate result, we derive Theorem \ref{OURS} which generalizes Theorem \ref{FFF}. We then obtain Theorems  \ref{THAFSCHAR} and \ref{SCHFRS}, they generalize Theorems \ref{HOLUBTHEOREM} and \ref{OLECHA}, respectively. Next we characterize all the dual frames in Theorem  \ref{ALLDUAL} whose particular case is Theorem \ref{LITHM}. Since there are two sequences in ASFs, we define similarity using two operators (Definition \ref{SIMILARITYMINE}). We then characterize similar frames using operators, which extends Theorem \ref{BALANCHARSIM} and  it is free from natural numbers (Theorem \ref{SEQUENTIALSIMILARITY}).   Later we introduce orthogonal ASFs (Definition \ref{ORTHOGONALDEF}) and derive an interpolation result (Theorem \ref{INTERPOLATION}). We end the paper by relating notions duality, similarity and orthogonality (Propositions \ref{LASTONE} and \ref{LASTTWO}).

 \section{Characterizations for subclasses of ASF, its dual frames and similarity of ASFs}\label{SECTIONTWO}
 Let $\mathcal{X}$ be a separable Banach space and $\mathcal{X}^*$ be its dual. In the rest of this paper, $\{e_n\}_n$ denotes the standard Schauder basis for  $\ell^p(\mathbb{N})$, $p \in [1, \infty)$ and $\{h_n\}_n$ denotes the coordinate functionals associated with $\{e_n\}_n$. Equation (\ref{REP}) motivated Casazza, Dilworth, Odell, Schlumprecht, and Zsak,  to define the notion of Schauder frame for $\mathcal{X}$ in 2008 \cite{CASAZZA}.
  \begin{definition}\cite{CASAZZA}\label{FRAMING}
  Let $\{\tau_n\}_n$ be a sequence in  $\mathcal{X}$ and 	$\{f_n\}_n$ be a sequence in  $\mathcal{X}^*.$ The pair $ (\{f_n \}_{n}, \{\tau_n \}_{n}) $ is said to be a Schauder frame for $\mathcal{X}$ if 
  \begin{align}\label{SFEQUA}
  x=\sum_{n=1}^\infty
  f_n(x)\tau_n, \quad \forall x \in
  \mathcal{X}.
  \end{align}
  \end{definition} 
To a greater extent, Definition \ref{FRAMING} was generalized by Freeman, Odell,  Schlumprecht, and Zsak in 2014 \cite{FREEMANODELL} (actually it was first generalized  for $\mathbb{R}^n$ by Thomas in 2012 \cite{THOMAS}).
   \begin{definition}\cite{FREEMANODELL, THOMAS}\label{ASFDEF}
  	Let $\{\tau_n\}_n$ be a sequence in  $\mathcal{X}$ and 	$\{f_n\}_n$ be a sequence in  $\mathcal{X}^*.$ The pair $ (\{f_n \}_{n}, \{\tau_n \}_{n}) $ is said to be an approximate Schauder frame (ASF) for $\mathcal{X}$ if 
  	\begin{align}\label{ASFEQUA}
 \text {(frame operator)}\quad	S_{f, \tau}:\mathcal{X}\ni x \mapsto S_{f, \tau}x\coloneqq \sum_{n=1}^\infty
  	f_n(x)\tau_n \in
  	\mathcal{X}
  	\end{align}
  	is a well-defined bounded linear, invertible operator.
  \end{definition} 
  Note that whenever $S_{f, \tau}=I_\mathcal{X}$, the identity operator on $\mathcal{X}$, Definition \ref{ASFDEF} reduces to Definition \ref{FRAMING}. Since $S_{f, \tau}$ is invertible, it follows that there are $a,b>0$ such that 
  \begin{align*}
  a\|x\|\leq \left\|\sum_{n=1}^\infty
  f_n(x)\tau_n \right\|\leq b\|x\|, \quad \forall x \in  \mathcal{X}.
  \end{align*} 
  We call $a$ as lower ASF bound and $b$ as upper ASF bound. Supremum (resp. infimum) of the set of all lower (resp. upper) ASF bounds is called optimal lower (resp. upper) ASF bound. From the theory of bounded linear operators between Banach spaces, one sees that optimal lower frame bound is $ \|S_{f, \tau}^{-1}\|^{-1}$ and optimal upper frame bound is $  \|S_{f,\tau}\|$. Advantage of ASF over Schauder frame is that it is more easier to get the operator in (\ref{ASFEQUA}) as invertible than obtaining Equation (\ref{SFEQUA}).
  \begin{definition}\label{PASFDEF}
  An ASF $ (\{f_n \}_{n}, \{\tau_n \}_{n}) $  for $\mathcal{X}$	is said to be p-ASF, $p \in [1, \infty)$ if both the maps 
  \begin{align}\label{ALIGNP}
  &\text{(analysis operator)} \quad \theta_f: \mathcal{X}\ni x \mapsto \theta_f x\coloneqq \{f_n(x)\}_n \in \ell^p(\mathbb{N}) \text{ and } \\
  &\text{(synthesis operator)} \quad\theta_\tau : \ell^p(\mathbb{N}) \ni \{a_n\}_n \mapsto \theta_\tau \{a_n\}_n\coloneqq \sum_{n=1}^\infty a_n\tau_n \in \mathcal{X}\label{AAA}
  \end{align}
  are well-defined bounded linear operators. A Schauder frame which is a p-ASF is called as a Parseval p-ASF.
  \end{definition}
  Observe that, in terms of inequalities, (\ref{ALIGNP}) and (\ref{AAA}) say that there exist $c,d>0$, such that 
  \begin{align*}
  &\left(\sum_{n=1}^\infty
  |f_n(x)|^p\right)^\frac{1}{p}\leq c \|x\|, \quad \forall x \in \mathcal{X} \text{ and } \\
  &\left\|\sum_{n=1}^\infty a_n\tau_n\right\|\leq d \left(\sum_{n=1}^\infty
  |a_n|^p\right)^\frac{1}{p}, \quad \forall \{a_n\}_n  \in \ell^p(\mathbb{N}).
  \end{align*}
  Now we have Banach space analogous of Theorem \ref{FFF}.
  \begin{theorem}\label{OURS}
  	Let $ (\{f_n \}_{n}, \{\tau_n \}_{n}) $ be a p-ASF for $\mathcal{X}$.  Then
  	\begin{enumerate}[\upshape(i)]
  		\item We have 
  	\begin{align}\label{REPBANACH}
  		x=\sum_{n=1}^\infty (f_nS_{f, \tau}^{-1})(x) \tau_n=\sum_{n=1}^\infty
  		f_n(x) S_{f, \tau}^{-1}\tau_n, \quad \forall x \in
  		\mathcal{X}.
  		\end{align}
  		\item $ (\{f_nS_{f, \tau}^{-1} \}_{n}, \{S_{f, \tau}^{-1} \tau_n \}_{n}) $ is a p-ASF for $\mathcal{X}$.
  		\item The analysis operator 
  		$
  		\theta_f: \mathcal{X} \ni x \mapsto \{f_n(x) \}_n \in \ell^p(\mathbb{N})
  		$
  		 is injective. 
  		 \item 
  		The synthesis operator 
  		$
  		\theta_\tau: \ell^p(\mathbb{N}) \ni \{a_n \}_n \mapsto \sum_{n=1}^\infty a_n\tau_n \in \mathcal{X}
  		$
  		is surjective.
  		\item Frame operator 
  		splits as $S_{f, \tau}=\theta_\tau\theta_f.$
  		\item  $P_{f, \tau}\coloneqq\theta_fS_{f,\tau}^{-1}\theta_\tau:\ell^p(\mathbb{N})\to \ell^p(\mathbb{N})$ is a projection onto   $\theta_f(\mathcal{X})$.
  	\end{enumerate}
  	\end{theorem}
  \begin{proof}
  	First follows from the continuity and linearity of $S_{f, \tau}^{-1}$.  Because $S_{f, \tau}$ is invertible, we have (ii). Again invertibility of  $S_{f, \tau}$ makes $\theta_f$  injective and $\theta_\tau$  surjective.  (v) and (vi) are routine calculations.
  \end{proof}
   Now we can derive a generalization of Theorem \ref{HOLUBTHEOREM} for Banach spaces.
  \begin{theorem}\label{THAFSCHAR}
  	A pair  $ (\{f_n\}_{n}, \{\tau_n\}_{n}) $ is a p-ASF for 	$\mathcal{X}$
  	if and only if 
  	\begin{align*}
  	f_n=h_n U, \quad \tau_n=Ve_n, \quad \forall n \in \mathbb{N},
  	\end{align*}  where $U:\mathcal{X} \rightarrow\ell^p(\mathbb{N})$, $ V: \ell^p(\mathbb{N})\to \mathcal{X}$ are bounded linear operators such that $VU$ is bounded invertible.
  \end{theorem}
  \begin{proof}
  	$(\Leftarrow)$ Clearly $\theta_f$ and $\theta_\tau$ are bounded linear operators. Now let $x\in \mathcal{X}$. Then 
  	\begin{align}\label{ORIGINALEQA}
  	S_{f, \tau}x= \sum_{n=1}^\infty
  	f_n(x)\tau_n=\sum_{n=1}^\infty h_n(Ux)Ve_n=V\left(\sum_{n=1}^\infty h_n(Ux)e_n\right)=VUx.
  	\end{align} 
  	Hence $S_{f, \tau}$ is bounded invertible. \\
  	$(\Rightarrow)$ Define $U\coloneqq \theta_f$, $V\coloneqq \theta_\tau$. Then $h_nUx=h_n\theta_fx=h_n(\{f_k(x)\}_k)=f_n(x)$, $\forall x \in \mathcal{X}$, $Ve_n=\theta_\tau e_n=\tau_n$, $\forall n \in \mathbb{N}$ and $VU=\theta_\tau \theta_f=S_{f, \tau}$ which is bounded invertible.
  \end{proof}
Note that Theorem \ref{THAFSCHAR}  generalizes Theorem \ref{HOLUBTHEOREM}. In fact, in the case of Hilbert spaces, Theorem \ref{THAFSCHAR} reads  as ``A sequence $\{\tau_n\}_n$ in  $\mathcal{H}$ is a
frame for $\mathcal{H}$	if and only if there exists a  bounded linear operator $T:\ell^2(\mathbb{N}) \to \mathcal{H}$ such that $Te_n=\tau_n$, for all $n \in \mathbb{N}$ and $TT^*$ is invertible". Now we know that $TT^*$ is invertible if and only if $T$ is surjective.\\
Since every separable Hilbert space admits an orthonormal basis, the existence of orthonormal basis in  Theorem \ref{OLECHA} is automatic. On the other hand, Enflo showed that there are  separable Banach spaces which do not  have Schauder basis (see \cite{JAMES}). Thus to obtain analogous of Theorem \ref{OLECHA} for  Banach spaces, we need to impose condition on $\mathcal{X}$.
\begin{theorem}\label{SCHFRS}
	Assume that $\mathcal{X}$ admits a Schauder basis $\{\omega_n\}_n$. Let $\{g_n\}_n$ denote the coordinate functionals associated with $\{\omega_n\}_n$. Assume that 
	\begin{align}\label{ASSUMPTIONNEEDED}
	 \{g_n(x)\}_n \in \ell^p(\mathbb{N}), \quad \forall x \in \mathcal{X}.
	\end{align} 
		Then a  pair  $ (\{f_n\}_{n}, \{\tau_n\}_{n}) $ is a p-ASF for 	$\mathcal{X}$
	if and only if 
	\begin{align*}
	f_n=g_n U, \quad \tau_n=V\omega_n, \quad \forall n \in \mathbb{N},
	\end{align*}  where $U, V:\mathcal{X} \rightarrow  \mathcal{X}$ are bounded linear operators such that $VU$ is bounded invertible.
\end{theorem}
\begin{proof}
	$(\Leftarrow)$ This is similar to the calculation done in (\ref{ORIGINALEQA}).\\
	$(\Rightarrow)$ Let $T$ be the map defined by 
	\begin{align*}
	T:\mathcal{X}\ni \sum_{n=1}^\infty a_n\omega_n\mapsto \sum_{n=1}^\infty a_ne_n\in \ell^p(\mathbb{N}).
	\end{align*}
	 Assumption (\ref{ASSUMPTIONNEEDED}) then says  that $T$ is bounded invertible operator with inverse $ T^{-1} :\ell^p(\mathbb{N}) \ni \sum_{n=1}^\infty b_ne_n \mapsto \sum_{n=1}^\infty b_n\omega_n \in \mathcal{X}$. Define $ U\coloneqq T^{-1}\theta_f$ and $V\coloneqq\theta_\tau T$. Then $ U,V$ are bounded  such that  $ VU=(\theta_\tau T)(T^{-1}\theta_f)=\theta_\tau\theta_f=S_{f,\tau}$ is invertible  and  for $ x \in \mathcal{X}$ we have 
	 \begin{align*}
	(g_nU)(x)&= g_n(T^{-1}\theta_fx)=g_n(T^{-1}(\{f_k(x)\}_{k})) 
	= g_n\left(\sum_{k=1}^\infty f_k(x)T^{-1}e_k\right)
	\\&=g_n\left(\sum_{k=1}^\infty f_k(x)\omega_k\right)
	=\sum_{k=1}^\infty f_k(x)g_n(\omega_k)=f_n(x), \quad \forall x \in \mathcal{X}
	 \end{align*} 
	 and $ V\omega_n=\theta_\tau T\omega_n=\theta_\tau e_n=\tau_n, \forall  n \in \mathbb{N}$.
\end{proof}
Equation (\ref{DUAL}) motivates us to define the notion of dual frame as follows.
\begin{definition}\label{SIMILARITYMINE}
Let $ (\{f_n\}_{n}, \{\tau_n\}_{n}) $ be a p-ASF for 	$\mathcal{X}$. 	A p-ASF $ (\{g_n \}_{n}, \{\omega_n \}_{n}) $ for $\mathcal{X}$ is a dual  for $ (\{f_n \}_{n}, \{\tau_n \}_{n}) $ if 
\begin{align*}
x=\sum_{n=1}^\infty g_n(x) \tau_n=\sum_{n=1}^\infty
f_n(x) \omega_n, \quad \forall x \in
\mathcal{X}.
\end{align*}
\end{definition}
Note that dual frames always exist. In fact, the  Equation (\ref{REPBANACH}) shows that the frame $ (\{f_nS_{f, \tau}^{-1} \}_{n}, \{S_{f, \tau}^{-1} \tau_n \}_{n}) $ is a  dual for $ (\{f_n\}_{n}, \{\tau_n\}_{n}) $. In the case of Hilbert spaces, the frame $\{S_\tau^{-1}\tau_n\}_n$ (Theorem \ref{FFF}) is called as canonical dual for $\{\tau_n\}_n$. Following this term, we call the  frame $ (\{f_nS_{f, \tau}^{-1} \}_{n}, \{S_{f, \tau}^{-1} \tau_n \}_{n}) $ as canonical dual for  $ (\{f_n\}_{n}, \{\tau_n\}_{n}) $. With this notion, the following theorem  follows easily.
\begin{theorem}
Let $ (\{f_n\}_{n}, \{\tau_n\}_{n}) $ be a  p-ASF for $ \mathcal{X}$ with frame bounds $ a$ and $ b.$ Then
\begin{enumerate}[\upshape(i)]
	\item The canonical dual p-ASF for the canonical dual p-ASF  for  $ (\{f_n\}_{n}, \{\tau_n\}_{n}) $ is itself.
	\item$ \frac{1}{b}, \frac{1}{a}$ are frame bounds for the canonical dual for $ (\{f_n\}_{n}, \{\tau_n\}_{n}) $.
	\item If $ a, b $ are optimal frame bounds for $ (\{f_n\}_{n}, \{\tau_n\}_{n}) $, then $ \frac{1}{b}, \frac{1}{a}$ are optimal  frame bounds for its canonical dual.
\end{enumerate} 
\end{theorem}
One can naturally ask when  a p-ASF has unique dual. An affirmative answer is given in the following result.
\begin{proposition}
	Let $ (\{f_n\}_{n}, \{\tau_n\}_{n}) $ be a  p-ASF for $ \mathcal{X}$. If $\{\tau_n\}_{n}$ is a Schauder basis for   $\mathcal{X}$ and $ f_k(\tau_n)=\delta_{k,n},\forall k,n \in \mathbb{N}$, then  $ (\{f_n\}_{n}, \{\tau_n\}_{n}) $ has unique dual.
\end{proposition}
\begin{proof}
	Let $ (\{g_n \}_{n}, \{\omega_n \}_{n}) $  and $ (\{u_n \}_{n}, \{\rho_n \}_{n}) $ be two dual p-ASFs for $ (\{f_n\}_{n}, \{\tau_n\}_{n}) $. Then 
	\begin{align*}
	\sum_{n=1}^\infty(g_n(x)-u_n(x))\tau_n=0= \sum_{n=1}^\infty f_n(x)(\omega_n-\rho_n), \quad \forall x \in \mathcal{X}.
	\end{align*}
 First equality gives $ g_n=u_n, \forall n \in \mathbb{N}$ and by evaluating second equality at a fixed $ \tau_k$ gives $ \omega_k=\rho_k$. Since $k$ was arbitrary, proposition follows.
\end{proof}
We now characterize dual frames  by using analysis and synthesis operators.
\begin{proposition}\label{ORTHOGONALPRO}
For two p-ASFs $ (\{f_n\}_{n}, \{\tau_n\}_{n}) $ and $ (\{g_n \}_{n}, \{\omega_n \}_{n}) $ for $\mathcal{X}$, the following are equivalent.
\begin{enumerate}[\upshape(i)]
	\item  $ (\{g_n \}_{n}, \{\omega_n \}_{n}) $ is a dual  for $ (\{f_n \}_{n}, \{\tau_n \}_{n}) $.
	\item $\theta_\tau\theta_g =\theta_\omega\theta_f =I_\mathcal{X}$.
\end{enumerate}
\end{proposition}
Like Lemmas \ref{LILEMMA1}, \ref{LILEMMA2} and Theorem \ref{LITHM}, we now  characterize dual frames using standard Schauder basis for $\ell^p(\mathbb{N})$.
  \begin{lemma}\label{ASFLEMMA1}
  	Let  $ (\{f_n \}_{n}, \{\tau_n \}_{n}) $  be a  p-ASF for   $\mathcal{X}$. Then a  p-ASF  $ (\{g_n \}_{n}, \{\omega_n \}_{n}) $ for $\mathcal{X}$ is a dual  for $ (\{f_n \}_{n}, \{\tau_n \}_{n}) $ if and only if
  	\begin{align*}
  	g_n=h_n U, \quad \omega_n=Ve_n, \quad \forall n \in \mathbb{N},
  	\end{align*} 
   where $ U:\mathcal{X} \rightarrow\ell^p(\mathbb{N})$ is  a bounded right-inverse of $ \theta_\tau$, and  $V: \ell^p(\mathbb{N}) \rightarrow \mathcal{X}$ is a bounded left-inverse of $ \theta_f$ such that $ VU$ is bounded invertible.
  \end{lemma}
  \begin{proof}
  	$(\Leftarrow)$ From the ``if" part of proof of Theorem \ref{THAFSCHAR}, we get that $ (\{g_n \}_{n}, \{\omega_n \}_{n}) $ is a p-ASF for $\mathcal{X}$. We have to check for duality of $ (\{g_n \}_{n}, \{\omega_n \}_{n}) $. For, $\theta_\tau\theta_g=\theta_\tau U=I_\mathcal{X} $, $ \theta_\omega\theta_f=V\theta_f =I_\mathcal{X}$.\\
  	$(\Rightarrow)$ Let $ (\{g_n \}_{n}, \{\omega_n \}_{n}) $ be a dual p-ASF for  $ (\{f_n \}_{n}, \{\tau_n \}_{n}) $.  Then $\theta_\tau\theta_g =I_\mathcal{X} $, $ \theta_\omega\theta_f =I_\mathcal{X}$. Define $ U\coloneqq\theta_g, V\coloneqq\theta_\omega.$ Then $ U:\mathcal{X} \rightarrow\ell^p(\mathbb{N})$ is a bounded right-inverse of $ \theta_\tau$, and  $V: \ell^p(\mathbb{N}) \rightarrow \mathcal{X}$ is  a bounded left-inverse of $ \theta_f$ such that the operator $ VU=\theta_\omega\theta_g=S_{g,\omega}$ is invertible. Further,
  	\begin{align*}
  	(h_nU)x=h_n\left(\sum_{k=1}^\infty g_k(x)e_k\right)=\sum_{k=1}^\infty g_k(x)h_n(e_k)=g_n(x), \quad \forall x \in \mathcal{X}
  	\end{align*} 
  	 and $Ve_n=\theta_\omega e_n=\omega_n, \forall n \in \mathbb{N} $.
  \end{proof}
  \begin{lemma}\label{ASFLEMMA2}
  Let $ (\{f_n \}_{n}, \{\tau_n \}_{n}) $ be a  p-ASF for   $\mathcal{X}$. Then 
  \begin{enumerate}[\upshape(i)]
  	\item $R: \mathcal{X} \rightarrow \ell^p(\mathbb{N})$ is a bounded right-inverse of $ \theta_\tau$  if and only if 
  	\begin{align*}
  	R=\theta_fS_{f,\tau}^{-1}+(I_{\ell^p(\mathbb{N})}-\theta_fS_{f,\tau}^{-1}\theta_\tau)U
  	\end{align*} where $U:\mathcal{X} \to \ell^p(\mathbb{N})$ is a bounded linear operator.
  	\item  $ L:\ell^p(\mathbb{N})\rightarrow \mathcal{X}$ is a bounded left-inverse of $ \theta_f$ if and only if 
  	\begin{align*}
  	L=S_{f,\tau}^{-1}\theta_\tau+V(I_{\ell^p(\mathbb{N})}-\theta_fS_{f,\tau}^{-1}\theta_\tau),
  	\end{align*} 
  	 where $V:\ell^p(\mathbb{N}) \to \mathcal{X}$ is a bounded linear operator. 
  \end{enumerate}		
  \end{lemma}
  \begin{proof}
  \begin{enumerate}[\upshape(i)]
  	\item  $(\Leftarrow)$  $\theta_\tau(\theta_fS_{f,\tau}^{-1}+(I_{\ell^p(\mathbb{N})}-\theta_fS_{f,\tau}^{-1}\theta_\tau)U)=I_\mathcal{X}+\theta_\tau U-I_\mathcal{X}\theta_\tau U=I_\mathcal{X}$. Therefore $\theta_fS_{f,\tau}^{-1}+(I_{\ell^p(\mathbb{N})}-\theta_fS_{f,\tau}^{-1}\theta_\tau)U$ is a bounded right-inverse of $ \theta_\tau$.  
  	
  	$(\Rightarrow)$  Define $U\coloneqq R $. Then $\theta_fS_{f,\tau}^{-1}+(I_{\ell^p(\mathbb{N})}-\theta_fS_{f,\tau}^{-1}\theta_\tau)U=\theta_fS_{f,\tau}^{-1}+(I_{\ell^p(\mathbb{N})}-\theta_fS_{f,\tau}^{-1}\theta_\tau)R=\theta_fS_{f,\tau}^{-1}+R-\theta_fS_{f,\tau}^{-1}=R$.
  	\item
  	$(\Leftarrow)$  $(S_{f,\tau}^{-1}\theta_\tau+V(I_{\ell^p(\mathbb{N})}-\theta_fS_{f,\tau}^{-1}\theta_\tau))\theta_f=I_\mathcal{X}+V\theta_f-V\theta_fI_\mathcal{X}=I_\mathcal{X}$. Therefore  $S_{f,\tau}^{-1}\theta_\tau+V(I_{\ell^p(\mathbb{N})}-\theta_fS_{f,\tau}^{-1}\theta_\tau)$ is a bounded left-inverse of $\theta_f$.
  	
  	$(\Rightarrow)$  Define $V\coloneqq L$. Then $S_{f,\tau}^{-1}\theta_\tau+V(I_{\ell^p(\mathbb{N})}-\theta_fS_{f,\tau}^{-1}\theta_\tau) =S_{f,\tau}^{-1}\theta_\tau+L(I_{\ell^p(\mathbb{N})}-\theta_fS_{f,\tau}^{-1}\theta_\tau)=S_{f,\tau}^{-1}\theta_\tau+L-S_{f,\tau}^{-1}\theta_\tau= L$.
  \end{enumerate}		
  \end{proof}
  \begin{theorem}\label{ALLDUAL}
  	   Let $ (\{f_n \}_{n}, \{\tau_n \}_{n}) $ be a  p-ASF for   $\mathcal{X}$. Then a  p-ASF  $ (\{g_n \}_{n}, \{\omega_n \}_{n}) $ for $\mathcal{X}$ is a dual  for $ (\{f_n \}_{n}, \{\tau_n \}_{n}) $ if and only if
  	   \begin{align*}
  	   &g_n=f_nS_{f,\tau}^{-1}+h_nU-f_nS_{f,\tau}^{-1}\theta_\tau U,\\
  	   &\omega_n=S_{f,\tau}^{-1}\tau_n+Ve_n-V\theta_fS_{f,\tau}^{-1}\tau_n, \quad \forall n \in \mathbb{N}
  	   \end{align*}
  	such that the operator 
  	\begin{align*}
  	S_{f,\tau}^{-1}+VU-V\theta_fS_{f,\tau}^{-1}\theta_\tau U
  	\end{align*}
  	is bounded invertible, where   $U:\mathcal{X} \to \ell^p(\mathbb{N})$ and $ V:\ell^p(\mathbb{N})\to \mathcal{X}$ are bounded linear operators.
  \end{theorem}
  \begin{proof}
  	Lemmas \ref{ASFLEMMA1} and  \ref{ASFLEMMA2} give the characterization of dual frame as 
  \begin{align*}
  	&g_n=h_n\theta_fS_{f,\tau}^{-1}+h_nU-h_n\theta_fS_{f,\tau}^{-1}\theta_\tau U=f_nS_{f,\tau}^{-1}+h_nU-f_nS_{f,\tau}^{-1}\theta_\tau U,\\
  	&\omega_n=S_{f,\tau}^{-1}\theta_\tau e_n+Ve_n-V\theta_fS_{f,\tau}^{-1}\theta_\tau e_n=S_{f,\tau}^{-1}\tau_n+Ve_n-V\theta_fS_{f,\tau}^{-1}\tau_n, \quad \forall n \in \mathbb{N}
  	\end{align*}
  	such that the operator 
  	$$(S_{f,\tau}^{-1}\theta_\tau+V(I_{\ell^p(\mathbb{N})}-\theta_fS_{f,\tau}^{-1}\theta_\tau))(\theta_fS_{f,\tau}^{-1}+(I_{\ell^p(\mathbb{N})}-\theta_fS_{f,\tau}^{-1}\theta_\tau)U) $$
  	is bounded invertible, where $U:\mathcal{X} \to \ell^p(\mathbb{N})$ and $ V:\ell^p(\mathbb{N})\to \mathcal{X}$ are bounded linear operators. By a direct expansion and simplification we get 
  	\begin{align*}
  	(S_{f,\tau}^{-1}\theta_\tau+V(I_{\ell^p(\mathbb{N})}-\theta_fS_{f,\tau}^{-1}\theta_\tau))(\theta_fS_{f,\tau}^{-1}+(I_{\ell^p(\mathbb{N})}-\theta_fS_{f,\tau}^{-1}\theta_\tau)U)
  	=S_{f,\tau}^{-1}+VU-V\theta_fS_{f,\tau}^{-1}\theta_\tau U.
  	\end{align*}
  \end{proof}
It is known that a bounded linear operator from $\ell^2(\mathbb{N}) $ to $\mathcal{H} $ is given by a Bessel sequence (for instance, see Theorem 3.1.3 in \cite{OLEBOOK}). Thus, for Hilbert spaces, Theorem \ref{ALLDUAL} becomes Theorem \ref{LITHM}.  Now we generalize Definition \ref{SIMILARDEFHILBERT} to Banach spaces.
  \begin{definition}
  	Two p-ASFs $ (\{f_n\}_{n}, \{\tau_n\}_{n}) $ and $ (\{g_n \}_{n}, \{\omega_n \}_{n}) $ for $\mathcal{X}$   are said to be similar or equivalent if there exist bounded invertible operators  $T_{f,g}, T_{\tau,\omega} :\mathcal{X} \to \mathcal{X}$ such that 
  	\begin{align*}
  	 g_n=f_nT_{f,g},\quad  \omega_n= T_{\tau,\omega}\tau_n, \quad \forall  n \in \mathbb{N}.
  	\end{align*}
  \end{definition}
Since the operators giving similarity are bounded invertible, the notion of similarity is symmetric. Further, a routine calculation shows that it is an equivalence relation (hence the name equivalent) on the set 
\begin{align*}
\{(\{f_n\}_{n}, \{\tau_n\}_{n}): (\{f_n\}_{n}, \{\tau_n\}_{n}) \text{ is a p-ASF for } \mathcal{X}\}.
\end{align*}
We now characterize similarity using just operators. In the sequel, given a p-ASF $ (\{f_n\}_{n}, \{\tau_n\}_{n}) $,  we set $P_{f, \tau}\coloneqq \theta_fS_{f,\tau}^{-1}\theta_\tau$. 
  \begin{theorem}\label{SEQUENTIALSIMILARITY}
  	For two p-ASFs $ (\{f_n\}_{n}, \{\tau_n\}_{n}) $ and $ (\{g_n \}_{n}, \{\omega_n \}_{n}) $ for $\mathcal{X}$, the following are equivalent.
  	\begin{enumerate}[\upshape(i)]
  		\item   $g_n=f_nT_{f, g} , \omega_n=T_{\tau,\omega}\tau_n,  \forall  n \in \mathbb{N}$, for some bounded invertible operators $T_{f,g}, T_{\tau,\omega}:\mathcal{X} \to \mathcal{X}.$ 
  		\item $\theta_g=\theta_f T_{f,g}, \theta_\omega=T_{\tau,\omega}\theta_\tau  $ for some bounded invertible operators $T_{f,g}, T_{\tau,\omega}:\mathcal{X} \to \mathcal{X}.$
  		\item $P_{g,\omega}=P_{f, \tau}.$
  	\end{enumerate}
  	If one of the above conditions is satisfied, then  invertible operators in  $\operatorname{(i)}$ and  $\operatorname{(ii)}$ are unique and are given by  $T_{f,g}= S_{f,\tau}^{-1}\theta_\tau\theta_g, T_{\tau, \omega}=\theta_\omega\theta_fS_{f,\tau}^{-1}.$ In the case that $ (\{f_n \}_{n}, \{\tau_n \}_{n}) $ is a Parseval p-ASF, then $ (\{g_n \}_{n}, \{\omega_n \}_{n}) $ is  a Parseval p-ASF if and only if $T_{\tau, \omega}T_{f,g} =I_\mathcal{X}$   if and only if $ T_{f,g}T_{\tau, \omega} =I_\mathcal{X}$. 
  \end{theorem}
  \begin{proof}
  	 (i) $\Rightarrow $ (ii) $ \theta_gx=\{g_n(x)\}_{n}=\{f_n(T_{f,g}x)\}_{n}=\theta_f(T_{f,g}x), \forall x \in \mathcal{X}$, 	 $ \theta_\omega(\{a_n\}_{n})=\sum_{n=1}^\infty a_n\omega_n=\sum_{n=1}^\infty a_nT_{\tau,\omega}\tau_n=T_{\tau,\omega}( \theta_\tau(\{a_n\}_{n})) , \forall \{a_n\}_{n} \in \ell^p(\mathbb{N})$.\\
  	 (ii) $\Rightarrow $ (iii) $  S_{g,\omega}= \theta_\omega\theta_g=T_{\tau,\omega} \theta_\tau\theta_f T_{f,g} =T_{\tau,\omega} S_{f, \tau}T_{f,g}$ and 
  	 \begin{align*}
  	  P_{g,\omega}=\theta_g S_{g,\omega}^{-1} \theta_\omega=(\theta_f T_{f,g})(T_{\tau,\omega} S_{f, \tau}T_{f,g})^{-1}(T_{\tau,\omega} \theta_\tau)= P_{f, \tau}.
  	 \end{align*}  
  (ii) $\Rightarrow $ (i)  $ \sum_{n=1}^\infty g_n(x)e_n=\theta_g(x)=\theta_f(T_{f,g}x)=\sum_{n=1}^\infty f_n(T_{f,g}x)e_n, \forall x \in \mathcal{X}.$ This clearly gives (i).\\
  	  (iii) $\Rightarrow $ (ii) $\theta_g=P_{g,\omega} \theta_g= P_{f,\tau}\theta_g=\theta_f(S_{f,\tau}^{-1}\theta_{\tau}\theta_g)$, and $\theta_\omega=\theta_\omega P_{g,\omega}=\theta_\omega P_{f,\tau}=(\theta_\omega\theta_fS_{f,\tau}^{-1})\theta_\tau $. We  show that $S_{f,\tau}^{-1}\theta_{\tau}\theta_g$ and $\theta_\omega\theta_fS_{f,\tau}^{-1} $ are invertible. For,
  	  \begin{align*}
  	  &(S_{f,\tau}^{-1}\theta_{\tau}\theta_g)(S_{g,\omega}^{-1}\theta_{\omega}\theta_f)=S_{f,\tau}^{-1}\theta_{\tau}P_{g,\omega}\theta_f=S_{f,\tau}^{-1}\theta_{\tau} P_{f,\tau}\theta_f=I_\mathcal{X},\\
  	  &(S_{g,\omega}^{-1}\theta_{\omega}\theta_f)(S_{f,\tau}^{-1}\theta_{\tau}\theta_g)=S_{g,\omega}^{-1}\theta_{\omega} P_{f,\tau}\theta_g=S_{g,\omega}^{-1}\theta_{\omega}P_{g,\omega}\theta_g=I_\mathcal{X} 
  	  \end{align*}
  	  and 
  	  \begin{align*}
  	  &(\theta_\omega\theta_fS_{f,\tau}^{-1})(\theta_\tau\theta_gS_{g,\omega}^{-1})=\theta_\omega P_{f,\tau}\theta_gS_{g,\omega}^{-1}=\theta_\omega P_{g,\omega}\theta_gS_{g,\omega}^{-1}=I_\mathcal{X},\\
  	  &(\theta_\tau\theta_gS_{g,\omega}^{-1})(\theta_\omega\theta_fS_{f,\tau}^{-1})=\theta_\tau P_{g,\omega}\theta_fS_{f,\tau}^{-1}=\theta_\tau P_{f,\tau}\theta_fS_{f,\tau}^{-1}=I_\mathcal{X}.
  	  \end{align*}    
  	Let $T_{f,g}, T_{\tau,\omega}:\mathcal{X} \to \mathcal{X}$ be bounded invertible and $g_n=f_nT_{f, g}, \omega_n=T_{\tau,\omega}\tau_n,  \forall  n \in \mathbb{N}$. Then $\theta_g=\theta_fT_{f, g} $ says that $\theta_\tau\theta_g=\theta_\tau\theta_fT_{f, g}=S_{f,\tau}T_{f, g}  $ which implies $ T_{f, g} =S_{f,\tau}^{-1}\theta_\tau\theta_g$, and $\theta_\omega=T_{\tau,\omega}\theta_\tau $ says $\theta_\omega\theta_f=T_{\tau,\omega}\theta_\tau\theta_f=T_{\tau,\omega}S_{f,\tau} $. Hence $T_{\tau,\omega}=\theta_\omega\theta_fS_{f,\tau}^{-1} $. 
  \end{proof}
  It is easy to see that, for Hilbert spaces,  Theorem \ref{SEQUENTIALSIMILARITY} reduces to Theorem \ref{BALANCHARSIM}. \\
  In the Hilbert space frame theory, there is a twin notion associated with dual frames which is known as orthogonal frames. This was first introduced by Balan in his Ph.D. thesis \cite{BALANTHESIS} and further studied by Han, and Larson \cite{HANMEMOIRS} (also see Chapter 6 in \cite{KORNELSON}). We now define the same notion for Banach spaces.
  \begin{definition}\label{ORTHOGONALDEF}
  	Let $ (\{f_n\}_{n}, \{\tau_n\}_{n}) $ be a p-ASF for 	$\mathcal{X}$. 	A p-ASF $ (\{g_n \}_{n}, \{\omega_n \}_{n}) $ for $\mathcal{X}$ is orthogonal   for $ (\{f_n \}_{n}, \{\tau_n \}_{n}) $ if 
  	\begin{align*}
  	0=\sum_{n=1}^\infty g_n(x) \tau_n=\sum_{n=1}^\infty
  	f_n(x) \omega_n, \quad \forall x \in
  	\mathcal{X}.
  	\end{align*}
  \end{definition}
  Unlike duality, the notion orthogonality is symmetric but not reflexive. Further, dual p-ASFs cannot be orthogonal to each other and orthogonal p-ASFs cannot be dual to each other. Moreover,  if $ (\{g_n\}_{n}, \{\omega_n\}_n)$ is orthogonal for $ (\{f_n\}_{n}, \{\tau_n\}_n)$, then  both $ (\{f_n\}_{n}, \{\omega_n\}_n)$ and $ (\{g_n\}_{n}, \{\tau_n\}_n)$ are not p-ASFs. Just like Proposition \ref{ORTHOGONALPRO} we have the following proposition.
  \begin{proposition}
  	For two p-ASFs $ (\{f_n\}_{n}, \{\tau_n\}_{n}) $ and $ (\{g_n \}_{n}, \{\omega_n \}_{n}) $ for $\mathcal{X}$, the following are equivalent.
  	\begin{enumerate}[\upshape(i)]
  		\item  $ (\{g_n \}_{n}, \{\omega_n \}_{n}) $ is  orthogonal  for $ (\{f_n \}_{n}, \{\tau_n \}_{n}) $.
  		\item $\theta_\tau\theta_g =\theta_\omega\theta_f =0$.
  	\end{enumerate}
  \end{proposition}
Usefulness of orthogonal frames is that we have interpolation result, i.e., these frames can be stitched along certain curves (in particular, on the unit circle centred at the origin) to get new frames.
\begin{theorem}\label{INTERPOLATION}
	Let $ (\{f_n\}_{n}, \{\tau_n\}_{n}) $ and $ (\{g_n \}_{n}, \{\omega_n \}_{n}) $ be  two Parseval p-ASFs for  $\mathcal{X}$ which are  orthogonal. If $A,B,C,D :\mathcal{X}\to \mathcal{X}$ are bounded linear  operators and  $ CA+DB=I_\mathcal{X}$, then  
	\begin{align*}
	(\{f_nA+g_nB\}_{n}, \{C\tau_n+D\omega_n\}_{n})
	\end{align*}
	 is a  Parseval p-ASF for  $\mathcal{X}$. In particular,  if scalars $ a,b,c,d$ satisfy $ca+db =1$, then 
	$ (\{af_n+bg_n\}_{n}, \{c\tau_n+d\omega_n\}_{n}) $ is a  Parseval p-ASF for  $\mathcal{X}$.
\end{theorem} 
\begin{proof}
	By a calculation we find  
	\begin{align*}
	\theta_{fA+gB} x = \{(f_nA+g_nB)(x) \}_{n}=\{f_n(Ax) \}_{n}+\{g_n(Bx) \}_{n}=\theta_f(Ax)+\theta_g(Bx), \quad \forall x \in \mathcal{X}
	\end{align*}
	 and 
	 \begin{align*}
	 \theta_{C\tau+D\omega}(\{a_n \}_{n})=\sum_{n=1}^\infty a_n(C\tau_n+D\omega_n)=C\theta_\tau(\{a_n \}_{n})+D\theta_\omega(\{a_n \}_{n}), \quad  \forall \{a_n\}_n  \in \ell^p(\mathbb{N}). 
	 \end{align*} 
	  So 	
	\begin{align*}
	S_{fA+gB,C\tau+D\omega} &=\theta_{C\tau+D\omega} \theta_{fA+gB}= ( C\theta_\tau+ D\theta_\omega)(\theta_fA+\theta_gB)\\
	&=C\theta_\tau\theta_fA+C\theta_\tau\theta_gB+D\theta_\omega\theta_fA+D\theta_\omega\theta_gB\\
	&=CS_{f,\tau}A+0+0+DS_{g,\omega}B
	=CI_\mathcal{X}A+DI_\mathcal{X}B=I_\mathcal{X}.
	\end{align*}
\end{proof} 
Using Theorem \ref{SEQUENTIALSIMILARITY} we finally relate three notions : duality, similarity and orthogonality.
  \begin{proposition}\label{LASTONE}
  	For every p-ASF $(\{f_n\}_{n}, \{\tau_n\}_{n})$, the canonical dual for $(\{f_n\}_{n}, \{\tau_n\}_{n})$ is the only dual p-ASF that is similar to $(\{f_n\}_{n}, \{\tau_n\}_{n})$.
  \end{proposition}
  \begin{proof}
  	Let us suppose that  two p-ASFs $(\{f_n\}_{n}, \{\tau_n\}_{n})$ and $ (\{g_n \}_{n}, \{\omega_n \}_{n}) $  are similar and dual to each other. Then there exist bounded invertible operators  $T_{f,g}, T_{\tau,\omega} :\mathcal{X}\to \mathcal{X}$  such that $ g_n=f_nT_{f,g},\omega_n=T_{\tau,\omega}\tau_n ,\forall n \in \mathbb{N}$. Theorem \ref{SEQUENTIALSIMILARITY} then gives
  	\begin{align*}
  	T_{f,g}=S_{f,\tau}^{-1}\theta_\tau\theta_g=S_{f,\tau}^{-1}I_\mathcal{X}=S_{f,\tau}^{-1}\text{ and }T_{\tau, \omega}=\theta_\omega\theta_fS_{f,\tau}^{-1}=I_\mathcal{X}S_{f,\tau}^{-1}=S_{f,\tau}^{-1}.
  	\end{align*} 
  	 Hence $ (\{g_n \}_{n}, \{\omega_n \}_{n}) $ is the canonical dual for  $(\{f_n\}_{n}, \{\tau_n\}_{n})$.	
  \end{proof}
  \begin{proposition}\label{LASTTWO}
  	Two similar  p-ASFs cannot be orthogonal.
  \end{proposition}
  \begin{proof}
  	Let $(\{f_n\}_{n}, \{\tau_n\}_{n})$ and $ (\{g_n \}_{n}, \{\omega_n \}_{n}) $  be two p-ASFs which are similar.  Then there exist bounded  invertible  operators $T_{f,g}, T_{\tau,\omega} :\mathcal{X}\to \mathcal{X}$  such that $ g_n=f_nT_{f,g},\omega_n=T_{\tau,\omega}\tau_n ,\forall n \in \mathbb{N}$.  Theorem \ref{SEQUENTIALSIMILARITY} then says   $\theta_g=\theta_f T_{f,g}, \theta_\omega=T_{\tau,\omega}\theta_\tau  $.    Therefore 
  	\begin{align*}
  	\theta_\tau \theta_g=\theta_\tau\theta_f T_{f,g}=S_{f,\tau}T_{f,g}\neq0. 
  	\end{align*}
  \end{proof}
We end the paper with a remark.
  \begin{remark}
  	For every p-ASF  $(\{f_n\}_{n}, \{\tau_n\}_{n}),$ both  p-ASFs
  	\begin{align*}
   (\{f_n{S}_{f, \tau}^{-1}\}_{n}, \{\tau_n\}_{n})	 \text{ and }  (\{f_n\}_{n}, \{{S}_{f, \tau}^{-1}\tau_n\}_{n})
  	\end{align*}   
  	are  Parseval p-ASFs and are  similar to  $(\{f_n\}_{n}, \{\tau_n\}_{n})$.  Therefore  each p-ASF is similar to  Parseval p-ASFs.
  \end{remark}

  \section{Acknowledgements}
  We thank Prof. Radu Balan, University of Maryland, USA for sending a copy of his Ph.D. thesis \cite{BALANTHESIS}. We also thank Prof. D. Freeman,  St. Louis University, USA for providing a copy of the reference \cite{THOMAS} to us. First author thanks National Institute of Technology (NITK) Surathkal for
  financial assistance.

 \bibliographystyle{plain}
 \bibliography{reference.bib}

\begin{thebibliography}{10}

\bibitem{BALAN}
Radu Balan.
\newblock Equivalence relations and distances between {H}ilbert frames.
\newblock {\em Proc. Amer. Math. Soc.}, 127(8):2353--2366, 1999.

\bibitem{BALANTHESIS}
Radu~Victor Balan.
\newblock {\em A study of {W}eyl-{H}eisenberg and wavelet frames}.
\newblock ProQuest LLC, Ann Arbor, MI, 1998.
\newblock Thesis (Ph.D.)--Princeton University.

\bibitem{CASAZZA}
P.~G. Casazza, S.~J. Dilworth, E.~Odell, Th. Schlumprecht, and A.~Zsak.
\newblock Coefficient quantization for frames in {B}anach spaces.
\newblock {\em J. Math. Anal. Appl.}, 348(1):66--86, 2008.

\bibitem{OLEBOOK}
Ole Christensen.
\newblock {\em Frames and bases: An introductory course}.
\newblock Applied and Numerical Harmonic Analysis. Birkh\"{a}user Boston, Inc.,
  Boston, MA, 2008.

\bibitem{DUFFIN}
R.~J. Duffin and A.~C. Schaeffer.
\newblock A class of nonharmonic {F}ourier series.
\newblock {\em Trans. Amer. Math. Soc.}, 72:341--366, 1952.

\bibitem{FREEMANODELL}
D.~Freeman, E.~Odell, Th. Schlumprecht, and A.~Zs\'{a}k.
\newblock Unconditional structures of translates for {$L_p(\mathbb{R}^d)$}.
\newblock {\em Israel J. Math.}, 203(1):189--209, 2014.

\bibitem{KORNELSON}
Deguang Han, Keri Kornelson, David Larson, and Eric Weber.
\newblock {\em Frames for undergraduates}, volume~40 of {\em Student
  Mathematical Library}.
\newblock American Mathematical Society, Providence, RI, 2007.

\bibitem{HANMEMOIRS}
Deguang Han and David~R. Larson.
\newblock Frames, bases and group representations.
\newblock {\em Mem. Amer. Math. Soc.}, 147(697):x+94, 2000.

\bibitem{HOLUB}
James~R. Holub.
\newblock Pre-frame operators, {B}esselian frames, and near-{R}iesz bases in
  {H}ilbert spaces.
\newblock {\em Proc. Amer. Math. Soc.}, 122(3):779--785, 1994.

\bibitem{JAMES}
Robert~C. James.
\newblock Bases in {B}anach spaces.
\newblock {\em Amer. Math. Monthly}, 89(9):625--640, 1982.

\bibitem{LI}
Shidong Li.
\newblock On general frame decompositions.
\newblock {\em Numer. Funct. Anal. Optim.}, 16(9-10):1181--1191, 1995.

\bibitem{THOMAS}
S.~M. Thomas.
\newblock Approximate {S}chauder frames for {$\mathbb{R}^n$, Masters Thesis,
  St. Louis University, St. Louis, MO}.
\newblock 2012.

\end{thebibliography}

\end{document}